\documentclass[11pt]{article}
\usepackage{epsfig,amsmath,amsfonts,amssymb,amsthm,graphics}

\newcommand\SET[2]{\left\lbrace{\;#1\;:\;#2\;}\right\rbrace}
\newcommand\Reals{{\mathbb R}}
\newcommand\Integers{{\mathbb Z}}
\newcommand\Naturals{{\mathbb N}}
\newcommand\absval[1]{\left|{#1}\right|}

\newtheorem{prop}{Proposition}[section]
\newtheorem{thm}[prop]{Theorem}

\newtheorem{cor}[prop]{Corollary}

\newcommand\mattwor[1]{\left[{{\begin{array}{rr}#1\end{array}}}\right]}
\newcommand\mattwoc[1]{\left[{{\begin{array}{cc}#1\end{array}}}\right]}
\newcommand\matthreer[1]{\left[{{\begin{array}{rrr}#1\end{array}}}\right]}
\newcommand\matthreec[1]{\left[{{\begin{array}{ccc}#1\end{array}}}\right]}

\newcommand\textbox[1]{\hbox{\rm#1}}

\title{The divided cell algorithm and the inhomogeneous Lagrange and
  Markoff spectra }

\author{Richard T. Bumby \\
 Rutgers, the State University of New Jersey, \\ Department of
  Mathematics, Hill Center, Busch Campus, \\ 
110 Frelinghuysen Road, \\
  Piscataway, NJ 08854-8019, USA \\
{\tt bumby@math.rutgers.edu} 
\and 
Mary E. Flahive \\
 Department of Mathematics, \\ 
Oregon State University, \\ Corvallis, OR 97331-4605, USA \\
{\tt flahive@math.oregonstate.edu} }
\date{{\bf MSC:} 11J70,11J06, 11H16}

\begin{document}
\pagenumbering{arabic}
\addtocounter{secnumdepth}{1}

\maketitle

\begin{abstract}
The divided cell algorithm was introduced by Delone in 1947 
to calculate the inhomogeneous minima of 
binary quadratic forms and
developed further by
E. S. Barnes and H. P. F. Swinnerton-Dyer  in the 1950s.
We show how advances of the past fifty years 
in both symbolic computation 
and our understanding of homogeneous spectra can be combined to 
make divided cells more useful for organizing information about
inhomogeneous approximation problems.
A crucial part of our analysis relies on work of  Jane Pitman, who 
related the divided cell algorithm to  
the regular continued fraction algorithm.  
In particular, the relation to continued fractions allows two divided
cells for the same problem to be compared without stepping through the
chain of divided cells connecting them.
\end{abstract}

\section{Preliminaries}
Notational conventions and a basic framework for working with
approximation problems are collected here for the convenience of the
reader.

Diophantine Approximation problems deal with finding where the
restriction of a function to a special subset is small.  In this
paper, the function will be defined on the plane $\Reals^2$ and the
subset will be the integer lattice $\Integers^2$.  The function will
take nonnegative values, so ``small'' will mean ``close to zero''.  It
is customary to work with the reciprocal of the original function, and
to freely treat $\infty=1/0$ as a number, since it will be a possible
value of the supremum of a set of nonnegative numbers.  

By working with the name denoting a
function instead of the traditional convention of denoting a sequence
with subscripts, we are allowed the uncluttered notation $\limsup f$ to
denote the infimum over all cofinite subsets $S$ of $\Naturals$ of the
supremum of the values of $f$ restricted to $S$.

The computation of the infimum of a function can often be organized by
endowing the domain with a \emph{partial order} for which the given
function is order preserving.  If the partial order has the property
that descending sequences are finite, then the infimum of the values
over the whole set is equal to the infimum over the set of
\emph{minimal points} for the partial order.  This is valuable when
the set of minimal points has special properties.  In particular, it
is usually possible to index the minimal points by the set of all
integers $\Integers$ so that a pair of adjacent minimal points has
some special property. Such a function defined on set of all integers
$\Integers$ will be called a \emph{chain}.

The integer lattice in $\Reals^2$ is identified with $\Integers^2$ by
giving it a basis.  Certain bases aid in the identification of the
minimal points.  These depend on the expression and have been called
\emph{reduced}.  Dually, the expression giving the function in terms
of a reduced basis has also been called reduced.

Families of related problems lead to a space of reduced bases, and the
study of all reduced bases for a single problem can be expressed in
terms of a dynamical system on this space.  This study will lead to
strong results when the underlying space is \emph{compact}.

Our emphasis here will be \emph{visual} with pictures of the plane
$\Reals^2$ including the lattice $\Integers^2$.  However, while a
basis for the lattice is used in the algebraic description of the
objects in the figure, other considerations may be used in the choice
of \emph{viewing coordinates}.
\section{The Markoff Spectrum}
Traditionally, as in \cite{cusickandflahive89}, the homogeneous
\emph{Markoff Spectrum} is the set of values
\begin{equation}
M(F)=\sup\SET{\frac{\sqrt{D(F)}}{\absval{F(x,y)}}}
{x,y\in\Integers, (x,y)\neq(0,0)}
\label{eq:defM}
\end{equation} 
where $F(x,y)=Ax^2+Bxy+Cy^2$ is an \emph{indefinite} binary quadratic
form of discriminant $D(F)=B^2-4AC$.  The quantity $M(F)$ is given as
a \emph{normalized inverted minimum}: \emph{normalized} to allow
natural comparison between values of different forms $F$;
\emph{inverted} to allow simpler expressions for interesting values in
the spectrum.  Those $F$ with $F(x,y)=0$ for integers $x$ and $y$ (not
both zero), as well as those taking arbitrarily small values, have
$M(F)=\infty$.  The interesting cases are those for which $M(F)$ is
finite.

Since $F$ is indefinite, it can be factored over $\Reals$. We write
\begin{equation}
F(x,y)=(a_0x+b_0y)(a_1x+b_1y),
\label{eq:factform}
\end{equation}
and introduce new variables $\xi=a_0x+b_0y$, $\eta=a_1x+b_1y$ to get
$F=\xi\eta$.  Then, $\sqrt{D(F)}=\absval{a_0b_1-a_1b_0}$.

In particular, the expression $F(x,y)$ is encoded by the
matrix 
\begin{equation}
A=\mattwoc{a_0&b_0\\a_1&b_1\\}.
\label{eq:hommat}
\end{equation}
Left multiplication by this matrix takes the column with components
$(x,y)$ to one with components $(\xi,\eta)$.  Thus, it gives a change
of variables between the arithmetic and geometric aspects of the study
of the values of $F$ on the integer lattice.  The rows of the matrix
are the coefficients in the factors of $F(x,y)$. As a
change-of-variables matrix, its columns give the $(\xi,\eta)$
coordinates of the generators of the lattice.

A change of basis in the lattice multiplies the matrix in
(\ref{eq:hommat}) on the right by an integer matrix of determinant
$\pm1$; scaling the factors of $F(x,y)$ multiplies on the left by a
real diagonal matrix.  The value of $M(F)$ is not changed by these
actions.

A \emph{visual} approach to the Markoff Spectrum must show the integer
lattice and the lines $a_ix+b_iy=0\ (i=0,1)$.  However, it is more
convenient to use $(\xi,\eta)$ as \emph{viewing coordinates} since a
\emph{fixed} $F$ will be studied using \emph{different bases} for the
integer lattice.  In practice, this may be modified by a change of
scale $(\xi,\eta)\to(a\xi,\eta/a)$ in order to bring different lattice
points into focus.  Because of this choice of viewing coordinates, the
lines where $F(x,y)=0$ will be called the \emph{axes} of $F$.  For
example, a picture of $F(x,y)=x^2-3y^2$ on the integer lattice uses
viewing coordinates $(\xi,\eta)$ with
$\xi=x+y\sqrt3,\eta=x-y\sqrt3$. Figure~\ref{fig:hom} shows this view
of $F=0$ (now just the coordinate axes), the lattice generated by
$(x,y)=(1,0)$ and $(x,y)=(0,1)$, and the lattice cell whose $(x,y)$
coordinates are $(0,0),(1,0),(1,1),(0,1)$.
\begin{figure}
 \begin{center}
  \includegraphics[height=3in]{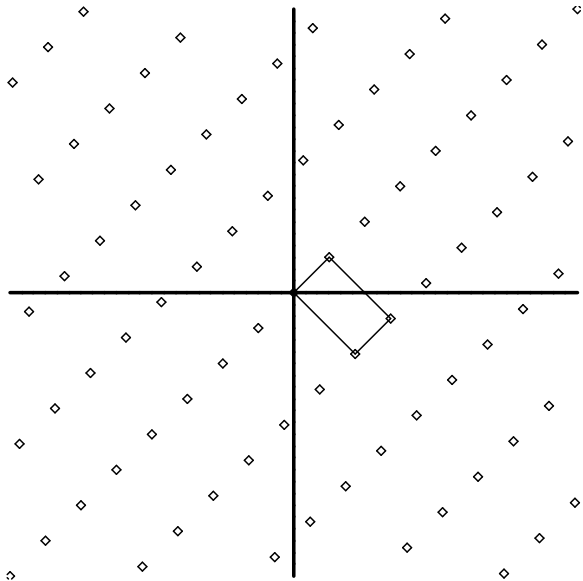}
\end{center}
  \caption{ \label{fig:hom} Homogeneous example}
\end{figure}

In computing $M(F)$, if a lattice point $P_0$ is closer to both axes
than the lattice point $P_1$ is, then $\absval{F}$ is smaller at $P_0$
than at $P_1$. This relation between $P_0$ and $P_1$ is a partial
order of the type mentioned in the Preliminaries.  Thus, only the
\emph{minimal points} for this partial order need be considered when
finding $M(F)$.  Arranging the minimal points in order of their
distance to a specified axis of $F$ gives a \emph{chain} of minimal
points (except when an axis contains a nonzero lattice point).  Since
these results are well known, and have been given in detail in
\cite{bumby1991}, features of this chain are only sketched here.  The
corresponding results for inhomogeneous problems will be described
later in more detail.  Figure~\ref{fig:hom} shows that $(x,y)=(1,0)$
and $(x,y)=(1,1)$ are minimal points, but $(x,y)=(0,1)$ is not since
$(1,0)$ is closer to both axes.

A full description shows that two successive minimal points always
generate the lattice.  These are the \emph{reduced bases} of the
lattice.  In a precise definition of a reduced basis, it is convenient
to fix the order of the axes, the order of the generators of the
lattice, and to choose between a generating vector and its negative.
With one set of choices, if $x$ and $y$ are the coordinates with
respect to a reduced basis, one has $a_0\geq a_1\geq0$ and $b_1\geq
-b_0\geq0$ in (\ref{eq:factform}) and (\ref{eq:hommat}).  Since the
matrix (\ref{eq:hommat}) determines the reduced basis, it is
appropriate to also speak of a \emph{reduced matrix} when these
conditions hold.  The inhomogeneous case will also require matrices
with $b_0\geq0\geq a_0$, but these are avoided in the tradition
treatment of the homogeneous case.

For \emph{three} consecutive minimal points, a matrix whose columns
are the basis consisting of the second and third points is the product
of the corresponding matrix for the first and second points with the
matrix
\begin{equation}
\mattwor{0&1\\1&-a\\}
\label{eq:stepmat}
\end{equation}
with $a=\left\lfloor{a_0/b_0}\right\rfloor$.  To restore
orientation and obtain the required signs of the matrix element, this
must be multiplied on the left by a diagonal matrix whose first
diagonal entry is positive and whose second entry is negative.  In the
example, when the matrix for the reduced basis $(1,1),(1,0)$
multiplied by the matrix in (\ref{eq:stepmat}) and rescaled, leads to
the equation
\begin{equation}
\mattwoc{1+\sqrt3&1\\1-\sqrt3&1\\}\mattwor{0&1\\1&-2\\}=
\mattwoc{-1+\sqrt3&0\\0&-1-\sqrt3\\}
\mattwoc{\frac{1+\sqrt3}{2}&1\\\frac{1-\sqrt3}{2}&1\\}.
\label{eq:exstep}
\end{equation}

The space of all reduced matrices with fixed determinant forms a
compact set.  Note that compactness requires that all inequalities be
\emph{inclusive}.  Classical work often aimed for unique
representations and required some inequalities to be strict, but
sacrificing uniqueness to have a compact space of reduced matrices
allows the Spectrum to be characterized in terms of \emph{attained}
extrema.

The chain of the matrices given by (\ref{eq:stepmat}) is one
description of the steps in the continued fraction algorithm.  It
produces a \emph{symbolic dynamics} that is useful for describing the
relation between the reduced bases of a given form and the computation
of $M(F)$.  In particular, a consistent choice of a vector from each
reduced basis leads to a chain of minimal points $(x_n,y_n)$, and
$M(F)=\sup M_n(F)$ where
\begin{equation}
M_n(F)=\frac{\sqrt{D(F)}}{\absval{F(x_n,y_n)}}.
\label{eq:local}
\end{equation}

Each index $n$ should be associated with both the minimal point
$(x_n,y_n)$ and the reduced basis with this point as first element.  A
sequence of indices can be found for which $M_n(F)\to M(F)$ and also
the corresponding reductions converge (see Lemma~6 of Chapter~1 of
\cite{cusickandflahive89}).  This shows that every value in the
Markoff Spectrum is an \emph{attained} supremum.  This result is known
as the \emph{Compactness Theorem} for the Markoff Spectrum.

A novel variation on this approach, allowing generalization to higher
dimensions, can be found in \cite{sensual}.

The \emph{Divided Cell Algorithm}  transfers these
properties of the continued fraction to inhomogeneous problems.

\section{The Lagrange Spectrum}
If the form $F(x,y)$ in (\ref{eq:factform}) is $x(y-x\alpha)$, then
$F(0,1)=0$, giving $M(F)=\infty$.  However, if $\alpha$ is irrational,
no other \emph{minimal points} $(x,y)$ have $F(x,y)=0$.  If
$0<\alpha<1$, we set $(x_{-1},y_{-1})=(1,0)$ and $(x_0,y_0)=(0,1)$,
giving a reduced basis; and then index the other minimal points by
positive integers.  Properties of rational approximations to $\alpha$
are determined by $L(\alpha)=\limsup M_n(F)$ for $n\in\Naturals$.  The
set of such values is called the \emph{Lagrange Spectrum}.  Theorem~1
of chapter~3 of \cite{cusickandflahive89} says that the Lagrange
Spectrum is a subset of the Markoff Spectrum.  This follows from the
proof of the Compactness Theorem for the Markoff Spectrum.  When
$L(\alpha)$ is finite, the forms appearing in the construction are all
equivalent to $F$, but the limiting form $F^*$ is nonzero on all
lattice points other than the origin and $L(\alpha)=M(F^*)$.

In this paper, we will concentrate on the inhomogeneous Markoff
Spectrum, but applications to the inhomogeneous Lagrange Spectrum will
follow by constructing a convergent sequence of reductions of the
inhomogeneous expression $x(y-x\alpha-\beta)$.

\section{The inhomogeneous Markoff Spectrum}
For \emph{inhomogeneous} problems, the form $F$ defined in
(\ref{eq:factform}) is replaced by
\begin{equation}
F_I(x,y)=(a_0x+b_0y+c_0)(a_1x+b_1y+c_1),
\label{eq:ifactform}
\end{equation}
while we continue to require $(x,y)\in\Integers^2$.  Figures
illustrating such problems will continue to be drawn in \emph{viewing
coordinates} for which $F_I(x,y)=0$ on the axes of the coordinate
system.  The origin of this coordinate system is no longer required to
be a lattice point.  For example, Figure~\ref{fig:inhom} modifies the
example of Figure~\ref{fig:hom} by using factors
$\xi_I=x+y\sqrt3-1-0.5\sqrt3$ and $\eta_I=x-y\sqrt3-1+0.5\sqrt3$ to
study the expression $F_I(x,y)=\xi_I\eta_I$.  Note that the origin is
now at $(x,y)=(1,0.5)$.  The parallelogram in the figure has vertices
whose $(x,y)$ coordinates are $(0,0),(1,1),(2,1)$, and $(0,1)$ with
edges that form a reduced basis.  Earlier work has required that
$(c_0,c_1)$ not be in the lattice generated by $(a_0,a_1)$ and
$(b_0,b_1)$ to explicitly exclude the homogeneous case.  We propose
to allow this case to \emph{exclude itself} because it necessarily has
a lattice point where $F_I(x,y)=0$ and interest will be centered on
those $F_I(x,y)$ that are bounded away from zero on the lattice.
\begin{figure}
\begin{center}
  \includegraphics[height=3in]{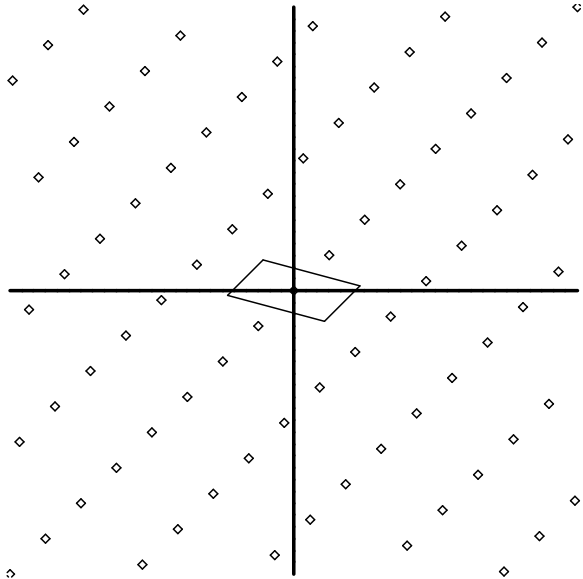}
\end{center}
  \caption{\label{fig:inhom} Inhomogeneous example}
\end{figure}
Notice that the parallelogram in Figure~\ref{fig:inhom} has one vertex
in each quadrant bounded by axes of $F_I$.  This property defines a
\emph{divided cell}: the word ``cell'' refers to a fundamental
parallelogram of the lattice, and it is ``divided'' by having its
vertices separated by the axes.

The definition of the inhomogeneous Markoff value is
\begin{equation}
M_I(F_I)=\sup\SET{\absval{\frac{a_0b_1-a_1b_0}{F_I(x,y)}}}
{x,y\in\Integers}
\label{eq:defMI}
\end{equation} 
using the notation of (\ref{eq:ifactform}).  Note that the origin of
the lattice is not excluded here, since it has no special role in
$F_I$.

To describe an inhomogeneous problem, the matrix of (\ref{eq:hommat})
must be replaced by  
\begin{equation}
B=\matthreec{a_0&b_0&c_0\\a_1&b_1&c_1\\}.
\label{eq:inhommat}
\end{equation}
Left multiplication by this matrix takes the column with components
$(x,y,1)$ to one with components $(\xi_I,\eta_I)$.  Again, the rows of
the matrix are the coefficients in the factors of $F_I(x,y)$.  The
interpretation of columns is a little different from the homogeneous
case: the third column is the image of the origin, and the first two
columns give generators of the lattice.  If the matrix is augmented
with a third row $[0\ 0\ 1]$, one gets the \emph{affine}
change-of-variables matrix relating the column with components
$(x,y,1)$ to one with components $(\xi_I,\eta_I,1)$.  In this matrix,
a column with $0$ in the third position represents a direction; one
with $1$ in the third position represents a point.  The vertices of
the cell are found by adding the sum of a subset of the first two
columns to the third column.  All such columns have $1$ in the third
position, so they can be expected to represent points.

Left multiplication by a two by two diagonal matrix changes the scale
on the axes and right multiplication by an integer matrix with a third
row $[0\ 0\ 1]$ and determinant $\pm1$ gives an affine change of basis
in the lattice.

\section{Divided cells as reduced objects}
We continue the convention of using the $(x,y)$ for coordinates in the
integer lattice, but describing geometric properties using
$(\xi_I,\eta_I)$ as viewing coordinates.

Divided cells will be the reduced objects for the study of $F_I(x,y)$
on the integer lattice.  Several proofs of the existence of divided
cells have been given, beginning with Delone \cite{del47} in 1947. Our
proof uses work of Pitman \cite{pitman58}, and will be given after
discussing the role of divided cells in Diophantine Approximation.

We choose the line $a_0x+b_0y+c_0=0$ to be the vertical axis in
Figure~\ref{fig:inhom} with the positive halfspace on the right.
Treating the vertices of the cell as the basic fundamental
parallelogram of $\Integers^2$ with $(0,0)$ in the lower left quadrant
of the figure, gives $c_0\leq0, b_0+c_0\leq0, a_0+c_0\geq0,
a_0+b_0+c_0\geq0$.  These inequalities imply $a_0\geq\absval{b_0}$.  A
similar analysis of $a_1x+b_1y+c_1=0$ as the horizontal line leads to
$b_1\geq\absval{a_1}$.  Conversely, these conditions on
$a_0,a_1,b_0,b_1$ give a nonempty set of possible solutions for
$c_0,c_1$.  Since $M_I(F_I)$ is invariant under scaling of the linear
factors of $F_I(x,y)$, one may introduce a convenient scaling, like
$a_0=b_1=1$.  Then $a_1$ and $b_0$ are each chosen from the interval
$[-1,1]$, and then each of $c_0$ and $c_1$ is chosen from an
appropriate closed intervals.  In this way, the space of divided cells
can be represented by the fourth power of a closed interval.  This
scaling will not be used in this paper, but we will insist that
$a_0>0$ and $b_1>0$, forcing the \emph{base vertex} to be in the third
quadrant.

This construction shows that the specification of a divided cell can
be done in two steps: first choose generators of the lattice giving
the directions of the sides of the cell; then locate the origin.
Barnes \cite{BarIV} introduced the term ``I-reduced'' for the lattice
bases arising in this way.  We keep the name, but take it to mean that
$a_0\geq\absval{a_1}$ and $b_1\geq\absval{b_0}$.

If $a_1b_0\leq0$, the cells are essentially the reduced cells of the
homogeneous case.  Such cells will be called \emph{Gaussian}, or
G-cells, indicating that they are reduced in the sense of Gauss.  Note
that, in contrast to the homogeneous case, no attempt is made to fix
the sign of $a_1$.

The cells that are \emph{not} Gaussian will be called N-cells. Since
definitions should use \emph{inclusive} inequalities, the correct
characterization of an N-cell is $a_1b_0\geq0$.  This allows
a cell to be both a G-cell and an N-cell, but only when one of its
sides is parallel to an axis.

If a parallelogram is an I-reduced cell, then the possible locations
of the origin in the cell form a rectangle inside the cell.  This
rectangle is called the \emph{inner box} (which we will sometimes call
simply a ``box'') of the I-reduced cell.  Figure~\ref{fig:cellbox}
shows two typical examples.  In the figure, a G-cell is on the left
and an N-cell is on the right. The width of the box is
$a_0-\absval{a_1}$, so that it degenerates to a vertical line segment
if $a_0=\absval{a_1}$.  Similarly, the inner box degenerates to a
horizontal line segment if $b_1=\absval{b_0}$.  When both
$a_0=\absval{a_1}$ and $b_1=\absval{b_0}$, the box is only a single
point.  This illustrates that the first row, which gives the
coefficients in the equation of the axis shown in the vertical
position and describes the first
coordinates of the cell, governs the divided cell step.

\begin{figure}
 \hbox{\includegraphics[height=2in]{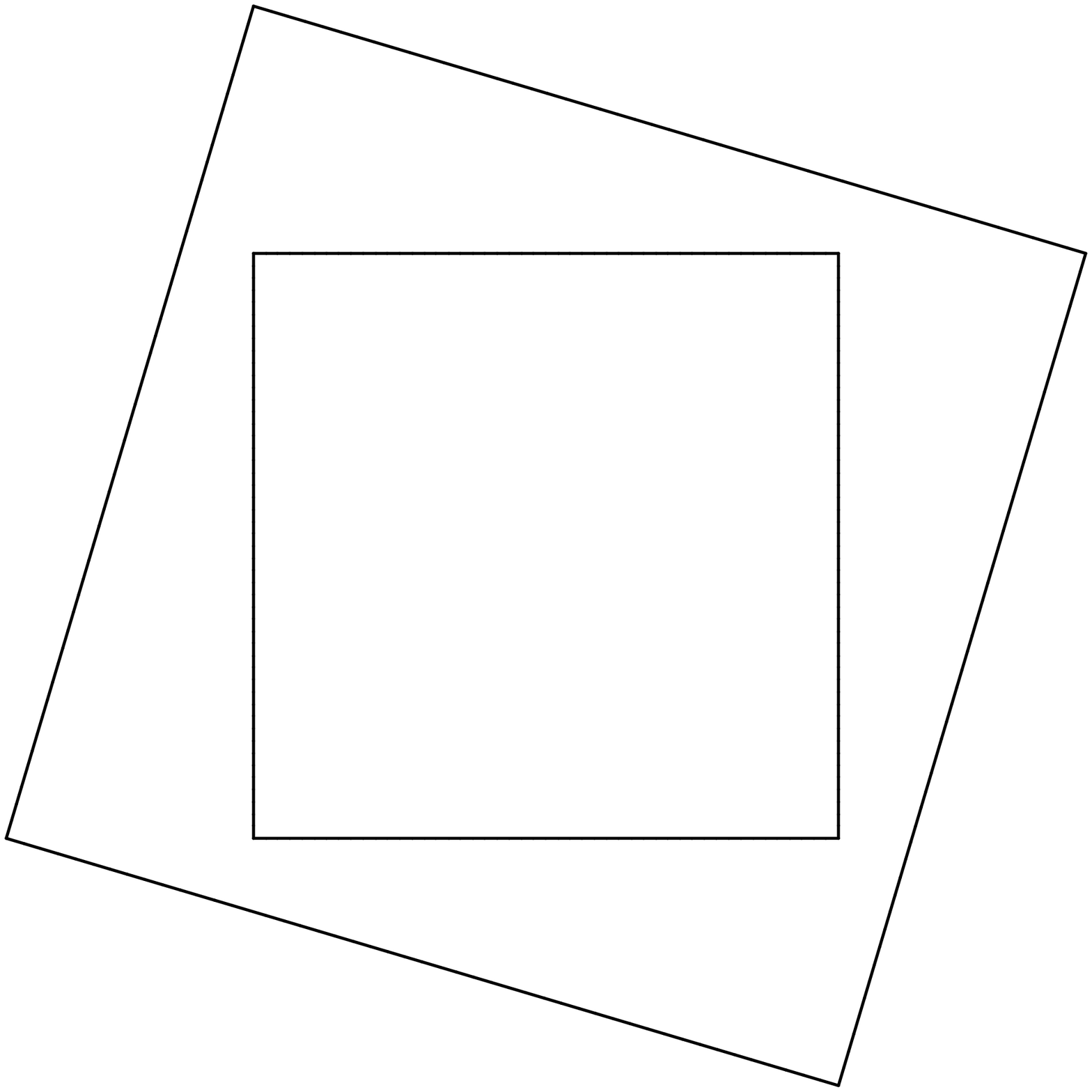}\hskip1in
             \includegraphics[height=2in]{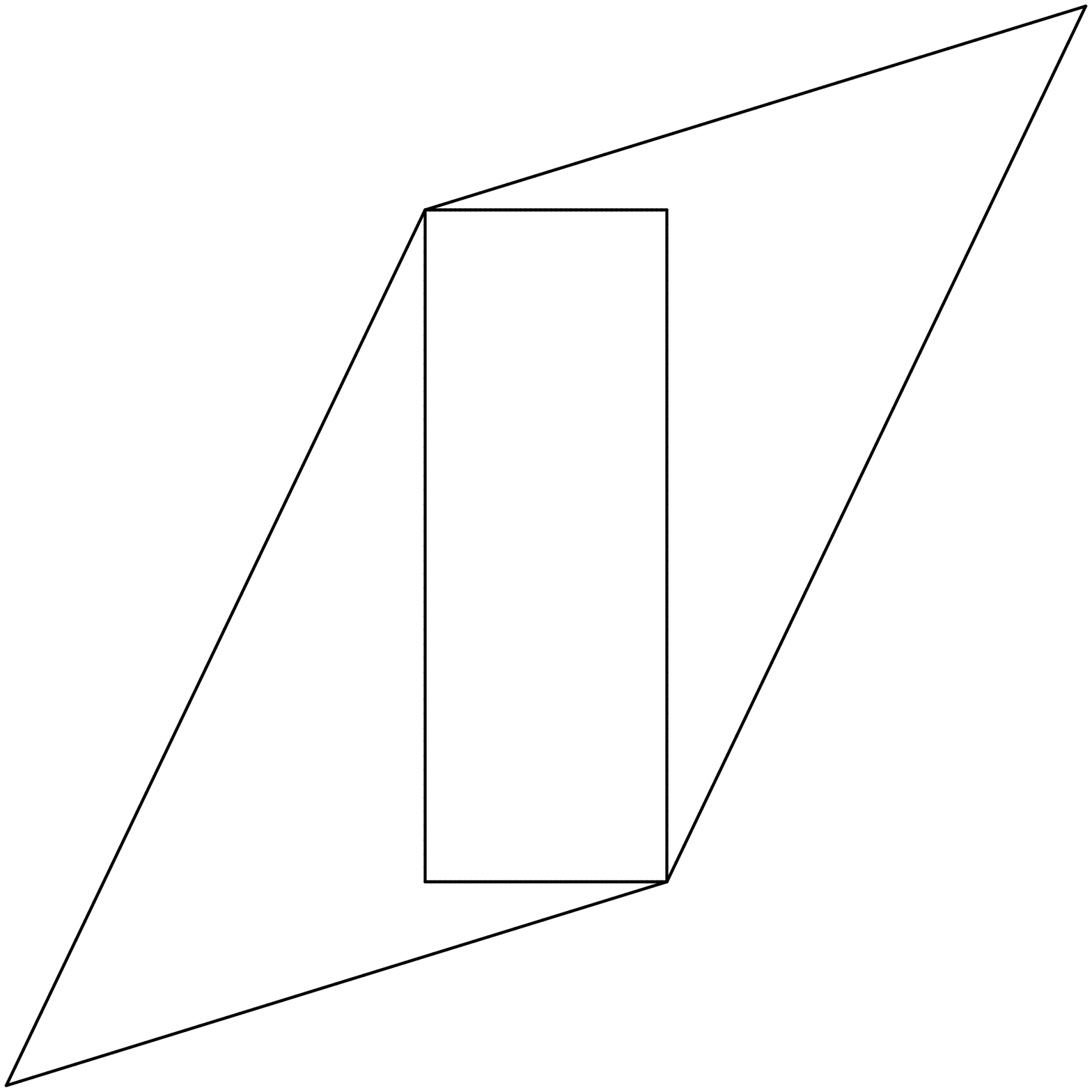}}
\caption{\label{fig:cellbox} Cells and Boxes}
\end{figure}

As in the homogeneous case, the partial order defining minimal points
considers distances to both axes.  However, this time it is necessary
to treat each quadrant separately.  Within a fixed quadrant, points
that are closer to both axes will be smaller points in the partial
order.  We call this the \emph{basic partial order}.
It will need to be modified, but this is a good
tentative definition.

\section{The divided cell algorithm}
Once one divided cell is available, it is possible to construct a
\emph{chain} of divided cells containing that cell. This construction
is the \emph{Divided Cell Algorithm}.  It must be shown that
$M_I(F_I)$ can be computed using only the vertices of the cells
obtained by this algorithm.  This is essentially the content of
Theorem~5 of \cite{BarS-DIII}.  Another approach to using divided
cells to compute $M_I(F_I)$ is given by the theorem on page 530 of
\cite{del47}.  Our proof will distinguish \emph{six} related chains
arising from the divided cell algorithm: a chain of cells, a chain of
boxes, and four chains of minimal points.  The relations among these
chains is not as direct as it is in the homogeneous case, so it is
useful to keep them separate while showing how they are related.  The
chains of minimal points --- one chain in each quadrant --- play a key
role in showing that all divided cells lie in a single chain and are
used to characterize the quantity $M_I(F_I)$ defined in
(\ref{eq:defMI}). Since the chains in different quadrants are
independent, distances in different quadrants may be weighted
differently.  We don't explore that here, but some consequences can be
found in Section~3 of \cite{BarS-DIII}.  Finally, the chain of boxes
shows the simplest progression from one axis to the other. All of
these chains terminate if there is a lattice direction parallel
to an axis, but our statements will make no effort to distinguish that
case.

The construction of the Divided Cell Algorithm is used in two
settings: given expression $F_I$, it produces its chain of divided
cells; given only an I-reduced basis for a lattice, it describes all
possible successor I-reduced bases.
\begin{thm}\label{thm:chain}
If $F_I$ admits one divided cell, then there is a chain of divided
cells containing that cell.  Given an I-reduced basis, there is one
shape of an N-cell arising as a successor and one possible
shape of a G-cell arising as a successor. The N-cell always occurs,
but the G-cell may not: the number of positions of the N-cell is
always one more than the number of positions of the G-cell.
\end{thm}
\begin{proof} 
Suppose that we are given a divided cell defined by a matrix as in
(\ref{eq:inhommat}).  Since this is a divided cell,  
$a_0\geq\absval{a_1}$ and $b_1\geq\absval{b_0}$, with additional
bounds on each $c_i$ in terms of $a_i$ and $b_i$.  
The details of the divided cell step depend on the sign of $b_0$.
The algorithm terminates if $b_0=0$, and there are only minor
differences between the other cases, so only the case of $b_0>0$ will
be illustrated.  

The definition of a divided cell then gives that $c_0<c_0+b_0\leq0\leq
c_0+a_0$, so that the line segment from $(c_0,c_1)$ to
$(c_0+b_0,c_1+b_1)$ forms the left side of the cell and crosses the
horizontal axis.  This side can be extended until it crosses the
vertical axis, giving an integer $h>0$ with $c_0+hb_0\leq0\leq
c_0+(h+1)b_0$.  The segment from $T_-\colon(c_0+hb_0,c_1+hb_1)$ to
$T_+\colon(c_0+hb_0+b_0,c_1+hb_1+b_1)$ will form the \emph{top} of the
next cell.  Similarly, the bottom of the next cell is found by
extending the right side to get a segment from
$B_-\colon(c_0+a_0-kb_0,c_1+a_1-kb_1)$ to
$B_+\colon(c_0+a_0+b_0-kb_0,c_1+a_1b_1-kb_1)$ for some $k>0$.  Since
segments $T_-T_+$ and $B_-B_+$ both cross the vertical axis, it
follows that $(h+k-1)b_0\leq a_0\leq(h+k+1)b_0$.  This analysis shows
that right multiplication by
\begin{equation}
S=\matthreec{0&-1&1\\1&h+k&-k\\0&0&1\\}(b_0>0)
\textbox{ or }S=\matthreec{0&1&0\\-1&h+k&1-h\\0&0&1\\}(b_0<0)
\label{eq:transmat}
\end{equation} 
gives the matrix representing the next cell.
In each case, $h$ and $k$ are positive integers with
$(h+k-1)\absval{b_0}\leq a_0\leq(h+k+1)\absval{b_0}$. If $a_0/b_0$ is
not an integer, $h+k$ must be one of the two integers nearest
to $\absval{a_0/b_0}$.  

The \emph{shape} of the successor cell is determined by $h+k$; and the
\emph{position} by $h$.  Those with different $h$ and the same $h+k$
are translates of one another The leftmost possible cell is an N-cell
with $h=1$ and $k$ as large as possible.  If $\absval{a_0/b_0}<2$,
this is the only successor.  Otherwise, the rules for determining the
inner box show that decreasing $k$ by $1$ and keeping $h$ fixed gives
a G-cell whose inner box abuts the box of this leftmost cell.  Then,
keeping this $k$ and increasing $h$ by $1$ gives a \emph{translate} of
the leftmost N-cell whose inner box abuts the box of this G-cell.  The
rightmost box will be an N-cell, and the the union of the inner boxes
of these possible successors covers the inner box of the original
cell.
\end{proof} 

Note that the first row of (\ref{eq:inhommat}), which gives the
coefficients in the equation of the axis shown in the vertical
position and describes the first coordinates of the cell, governs the
divided cell step.

The first part of Theorem~\ref{thm:chain} is illustrated in
Figure~\ref{fig:cellsboxes} showing a divided cell with its box and,
in two separate graphs, two successor cells with their boxes.  In this
picture, the original cell is a G-cell, and both types of successor
are shown with the G-cell on the left (note that this figure contains
G-cells with different signs of $a_1$).  Several lattice points are
also included.
\begin{figure}
 \hbox{\includegraphics[height=2in]{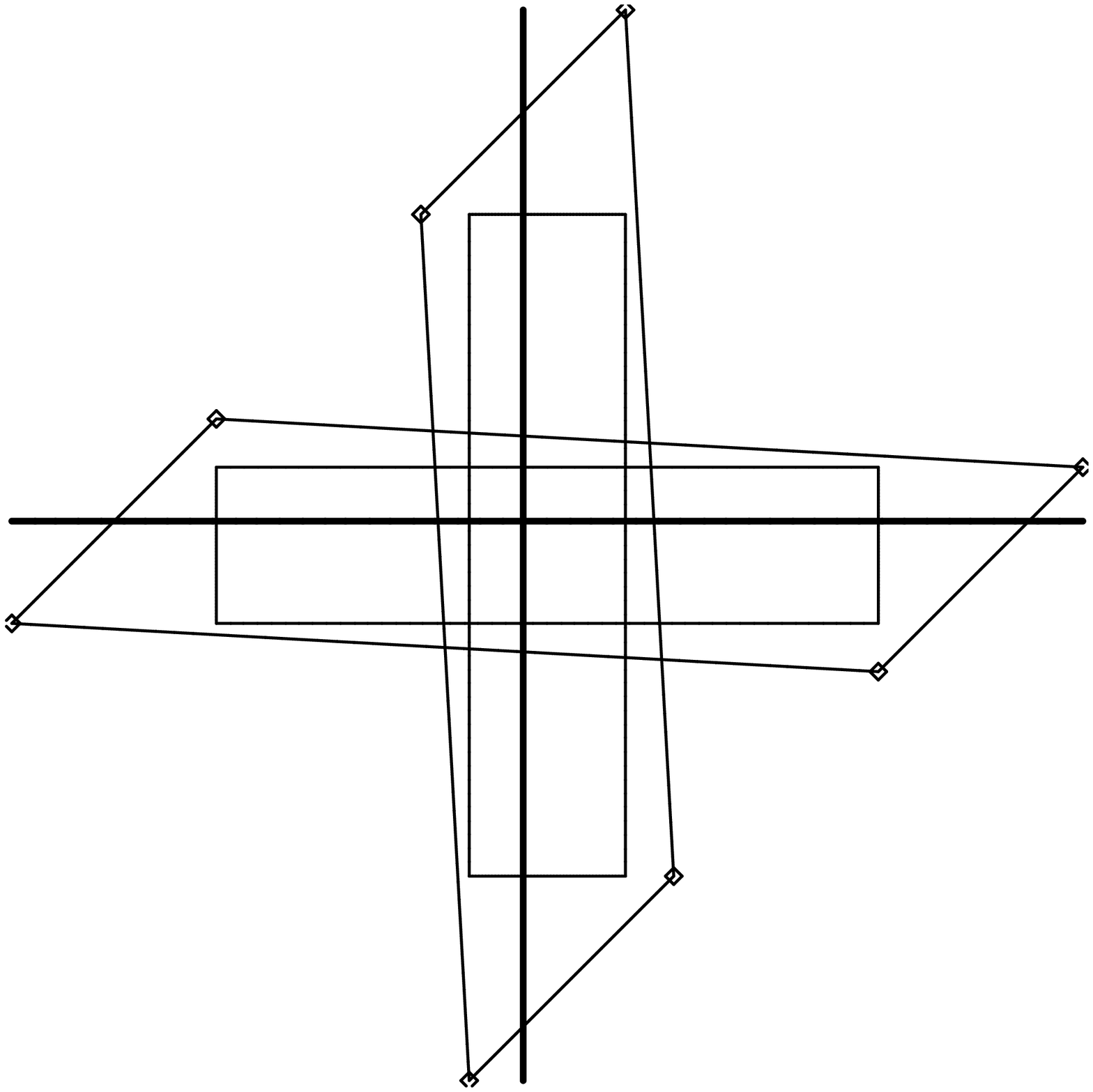} \hskip1in
             \includegraphics[height=2in]{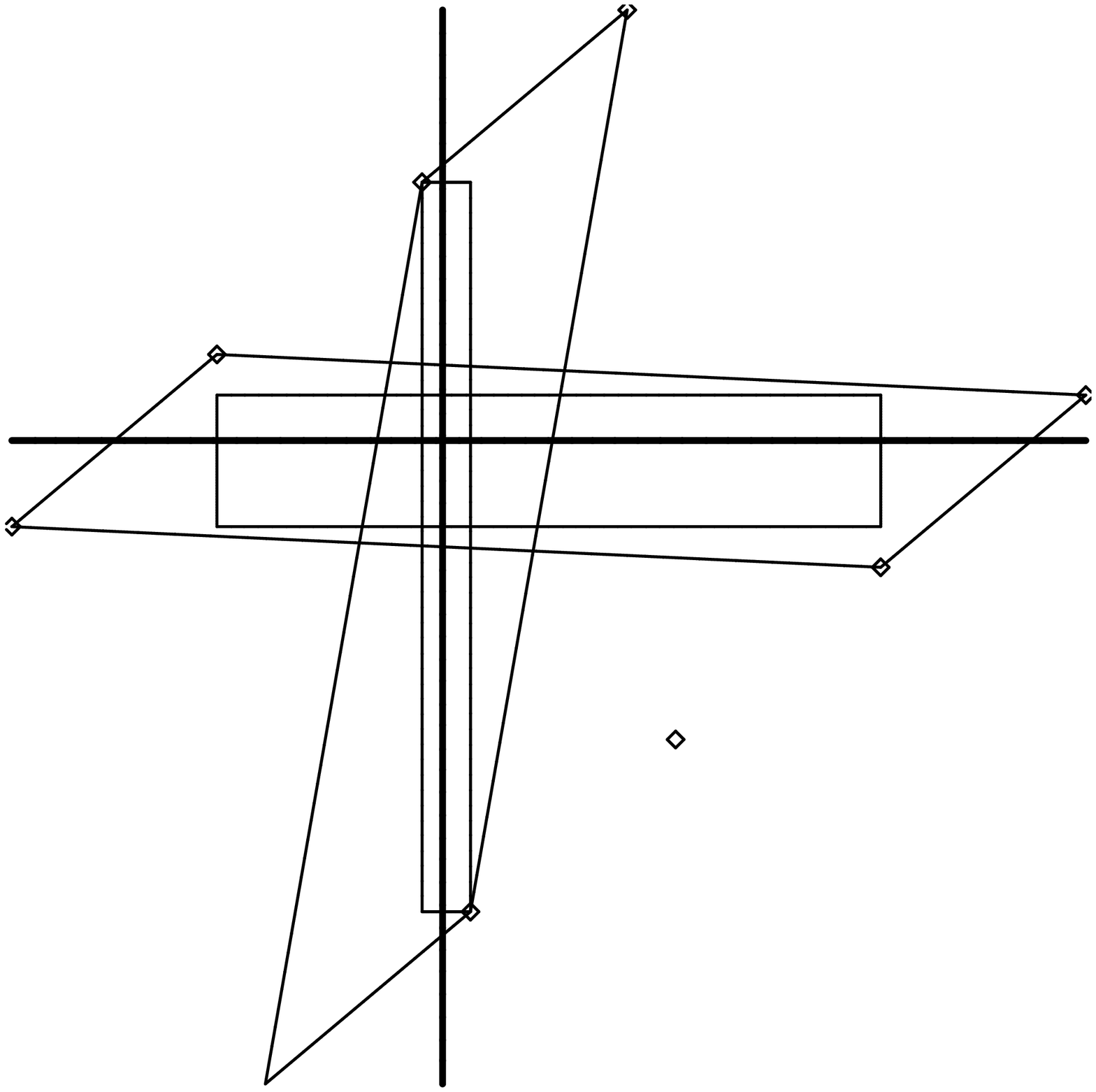}}
\caption{\label{fig:cellsboxes} Successor Cells and Boxes}
\end{figure}

Figure~\ref{fig:cellsboxes} may also be used to analyze the chain of
divided cell vertices in each quadrant.  In the pictures, the common
lattice direction of a cell and its successor gives a line joining the
vertices of those cells in the second quadrant, and also in the fourth
quadrant.  Moreover, these two lines are adjacent lattice lines in
that direction, so that there are no lattice points interior to the
strip bounded by those lines.  However, in the fourth quadrant of the
second picture there is a point on one of these
lines that is not a vertex of a divided cell although it meets our
preliminary requirement for being a minimal point.  We will now
resolve this difficulty.  We state the theorem for the first quadrant
in order to have names for the edges that we use, and concentrate on
points with small first coordinate  but the proof is
readily applied to both coordinates in all quadrants.

\begin{thm}\label{thm:linord}Given a divided cell $C$, let $I$ be the
  projection of the open top edge of $C$ on the horizontal axis. For
  each lattice line $L$ parallel to this edge, let $I_L$ be the points
  on $L$ whose projection on the horizontal axis lies in $I$. Then
  each $I_L$ contains at most one lattice point, and only those above
  the top edge of $C$ contain such a lattice point in the first
  quadrant. Furthermore, the projections onto the vertical axis of the
  $I_L$ are disjoint, so the ordering of these points by their second
  coordinate is the same as the order on the line $L$ containing the
  point.  
\end{thm}
\begin{proof} 
For a lattice line $L$, the distance between consecutive lattice
points on $L$ is fixed, and the top edge of $C$ gives one example of
such a pair of consecutive lattice points.  Again, since the lines are
parallel, the difference of first coordinates is also fixed and equal
to the width of $I$ in this case.  Hence, except when the endpoints of
$I_L$ are lattice points, there is a unique lattice point in each
$I_L$.

Similarly, the relation between projections on the vertical axis of
two consecutive $I_L$ is also fixed, so it will be the same as the
relation between the projections of the top and bottom edges of $C$.
However, $C$ is a divided cell, so all points on the top edge have
positive second coordinate and all points on the bottom edge have
negative second coordinate, so the projections of these edges are
disjoint.  For $L$ below the top edge of $C$, all points of $I_L$ have
negative second coordinate, so $I_L$ contains no point in the first
quadrant.
\end{proof} 

Any point in the first quadrant that is closer to the vertical axis
than the vertex $P$ of $C$ in that quadrant must project into $I$, but
Theorem~\ref{thm:linord} shows that all lattice points with that
property have larger second coordinate than $P$.  Hence $P$ is a
minimal point.  

When a side of $C$ is extended to meet the positive vertical axis, one
obtains a lattice line with one lattice point on each line parallel to
the top edge of $C$.  When the left side of the $C$ is used in this
construction, the first description of the divided cell step shows
that the first lattice point in the first quadrant is a vertex of the
successor divided cell.  If it is the extension of the right side of
$C$ that meets the positive vertical axis, the first several lattice
points will be in the first quadrant, but only the first and last of
these are vertices of divided cells.  This bypassing of minimal points
in the divided cell algorithm is easily accommodated by augmenting the
basic partial order.

\begin{thm}\label{thm:convex}
 If a line meets a quadrant in a bounded interval, the product of the
 distances to the axes is zero at the endpoints of the interval and
 has a unique interior maximum.  The distance decreases as one moves
 from the location of the maximum towards either axis. 
\end{thm}
\begin{proof} 
A calculus exercise!  When expressed in terms of one of the
coordinates the distance is a quadratic polynomial with negative
coefficient of the second degree term.
\end{proof} 

When the line in Theorem~\ref{thm:convex} is a lattice line, this says
that we may modify the basic partial order to also say that a lattice
point is greater than another lattice point on the line that is on the
same side of the point of maximum value of $F_I$ and farther from that
point.  With this modification, the only minimal points on the line in
this quadrant are the vertices of the original divided cell and its
successor.

Augmenting the basic partial order in this way on \emph{every} lattice
  line and forming the transitive closure gives a new partial order
  called the \emph{extended partial order} with fewer minimal points.
  We will say the $P$ is \emph{nearer} that $Q$ if $P\leq Q$ in the
  extended partial order.

\begin{thm}\label{thm:minpt} If divided cells exist, every minimal
  lattice point for the extended partial order is a vertex of a
  divided cell.
\end{thm}
\begin{proof} This is now little more than using a known divided cell
  as the basis and using previous results of this section for an
  induction step.  By symmetry, it suffices to show the result for
   a minimal lattice point $P$ in the first quadrant that is closer to the
  vertical axis.  By Theorem~\ref{thm:linord} there are finitely many
  minimal lattice points whose second coordinate lies between that of
  $P$ and the original divided cell vertex in this quadrant.  By the discussion
  following Theorem~\ref{thm:convex}, the first of these is the vertex of
  a divided cell.  There are fewer minimal lattice points in the first
  quadrant between this cell and the selected point, allowing
  induction to work.

\end{proof} 

We illustrate the second part of Theorem~\ref{thm:chain} with
Figure~\ref{fig:allboxes} showing the inner boxes of the possible
successors.  To draw both Figure~\ref{fig:cellsboxes} and
Figure~\ref{fig:allboxes}, we used $a_0/a_1=2+\sqrt5\approx4.236$. The
two parts of Figure~\ref{fig:allboxes} show the number of each type of
cell predicted by Theorem~\ref{thm:chain}. To avoid clutter, the cells
are not shown in Figure~\ref{fig:allboxes}, but the vertices are.
Note that lattice points appear as vertices of the inner box of an
N-cell, but the inner box of a G-cell is strictly interior to the cell
and contains no lattice point.  The cells that are collected in each
of the pictures in Figure~\ref{fig:allboxes} are translates of one
another in agreement with the expressions for their vertices appearing
in the proof of Theorem~\ref{thm:chain}.
\begin{figure}
\hbox{\includegraphics[height=2in]{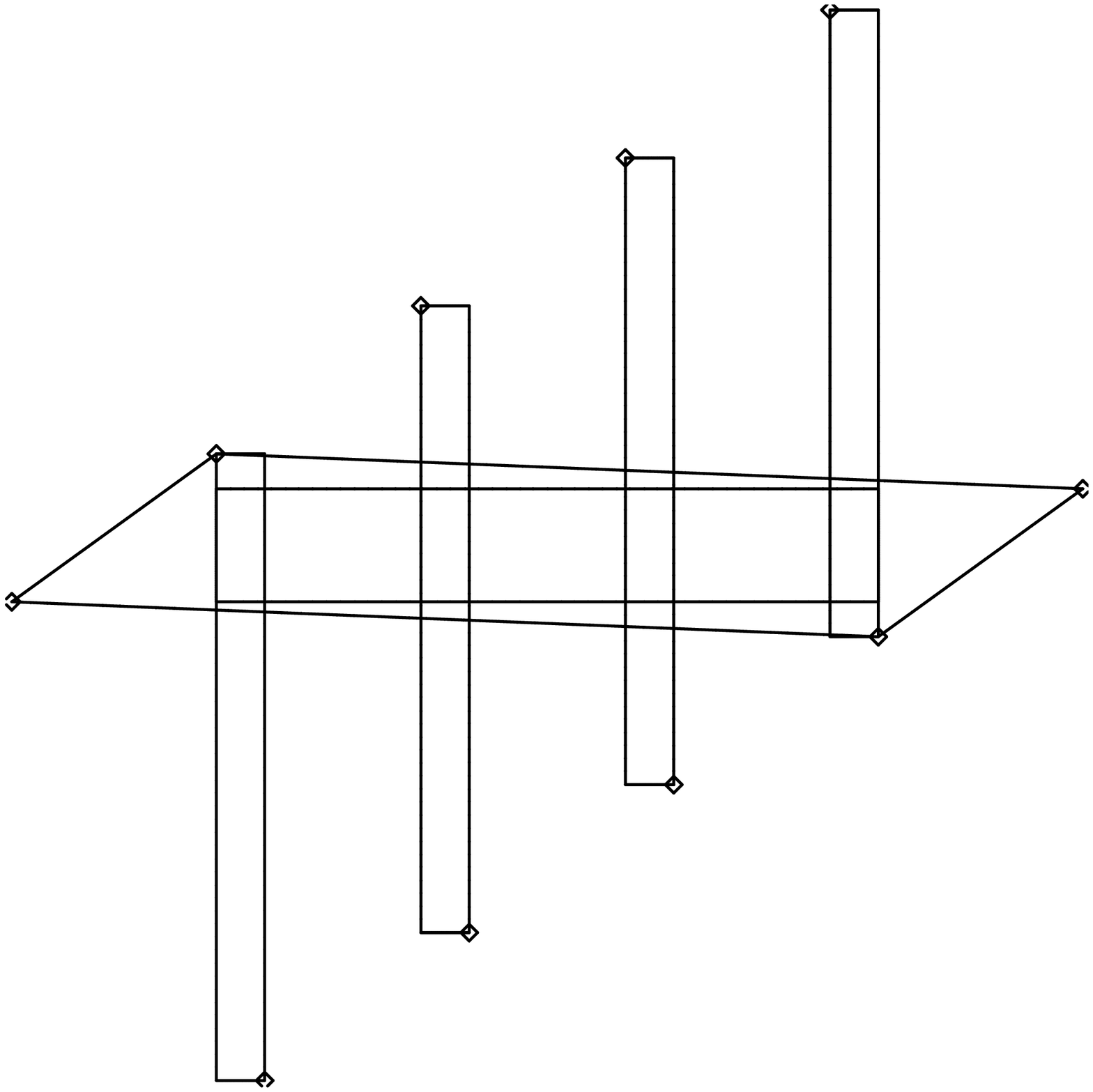} \hskip1in
             \includegraphics[height=2in]{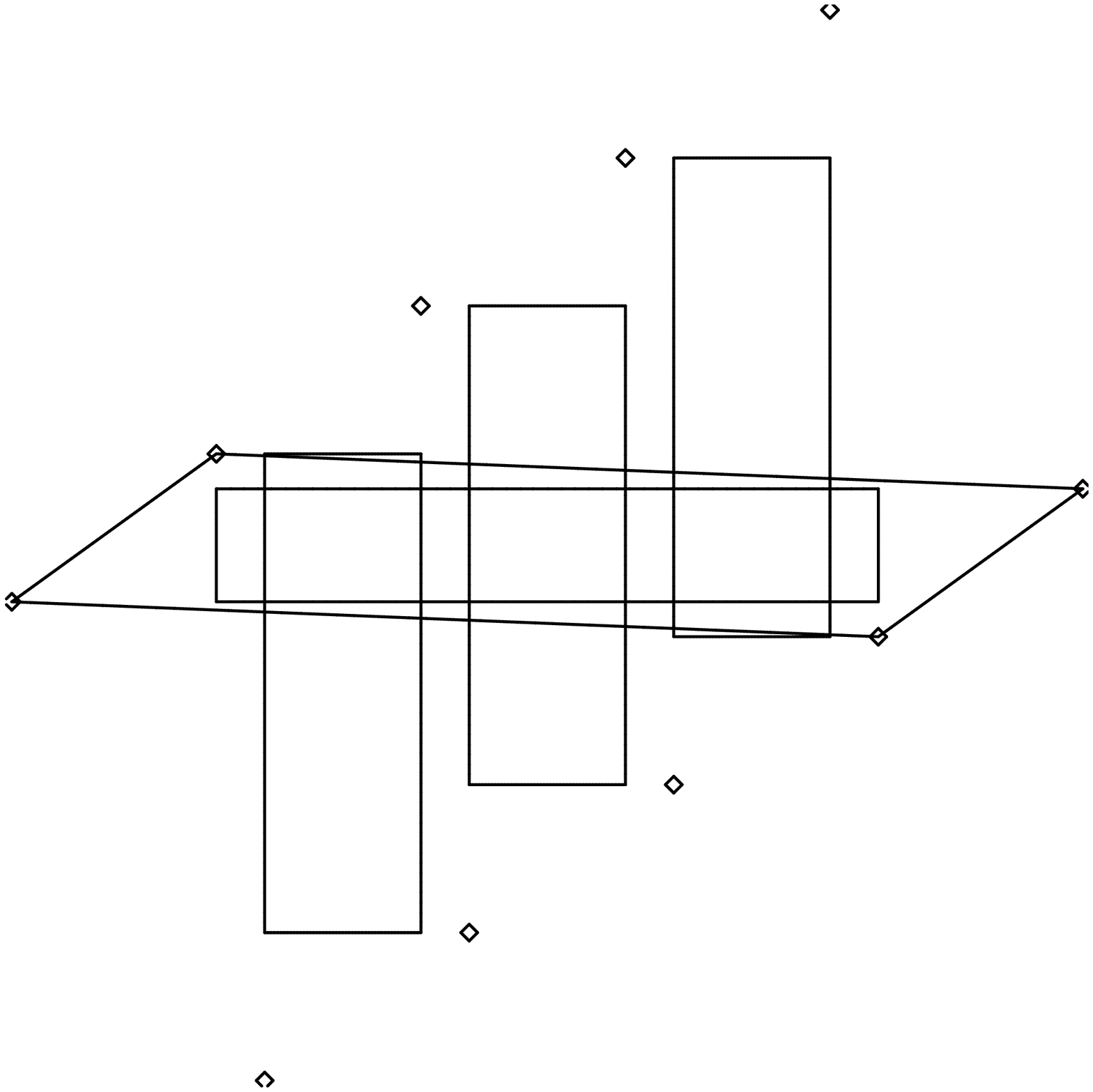}}
\caption{\label{fig:allboxes} All Successor Boxes}
\end{figure}
What Figure~\ref{fig:allboxes} shows is that the inner boxes of two
successor cells intersect in at most a vertical segment and that the
union of all of these boxes covers the inner box of the original
cell.

Combining this with the proof of Theorem~\ref{thm:minpt} we find that
the chain of boxes shows a systematic increase of height and decrease
of width as we step through the chain.

\section{Pitman's Theorem}Jane Pitman \cite{pitman58} related divided
cells, which are the reduced cells of an inhomogeneous approximation
problem, to the reduced bases of the corresponding homogeneous problem
given by the continued fraction algorithm.  Consequences of her work
are an easy proof of the existence of divided cells (given here as the
Corollary to Theorem~\ref{thm:Pitman}) and tools for recognizing
minimal points.

\begin{thm}\label{thm:Pitman}The cell of a Gaussian reduced form 
gives rise to two I-reduced N-cells.  If the Gaussian cell has $a_1\geq0$,
then the matrices corresponding to the other cells are obtained by
multiplying its matrix by
\begin{equation}
\matthreer{1&0&0\\1&1&0\\0&0&1\\}
\textbox{ or }
\matthreer{1&-1&1\\0&1&0\\0&0&1\\}.
\label{eq:neighborplus}
\end{equation} 
The union of the inner boxes of these three cells is (apart from
duplication on the boundary) a fundamental domain for the lattice.
\end{thm}
\begin{proof} Figure~\ref{fig:threebox} gives a ``proof without
  words''. It illustrates how the fundamental domain of the given
  Gaussian cell may be cut into pieces that may be translated and
  reassembled to form the union of the three boxes described in the
  statement of the theorem. The parallelogram whose sides are not
  horizontal or vertical is the given cell. The box in the center of
  the figure is the inner box of this cell.  The other boxes are the
  inner boxes of N-cells described in the statement of the theorem.
  The dashed line divides the part of the original cell outside the
  boxes into pieces congruent to the portions of the boxes outside the
  cell.
\end{proof} 
\begin{figure}
  \begin{center}
  \includegraphics[height=3in]{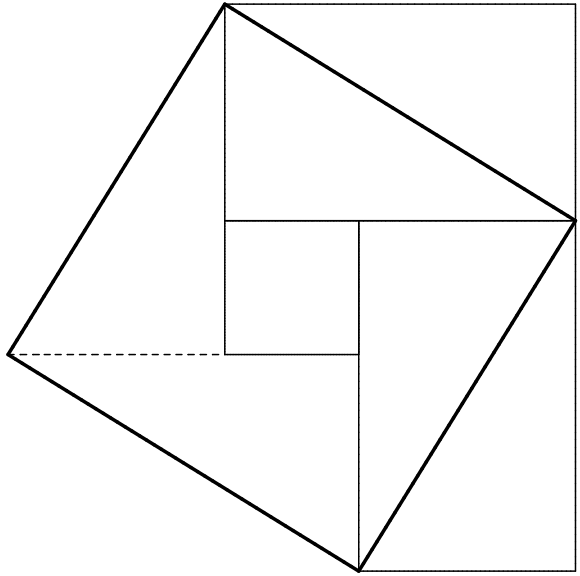}
\end{center}
  \caption{\label{fig:threebox} Three boxes form a fundamental domain}
\end{figure}

\begin{cor}Every linear inhomogeneous problem in
  $\Reals^2$ has divided cells.
\end{cor}
\begin{proof} Employ the homogeneous theory to reduce the linear part
  $(a_0x+b_0y)(a_1x+b_1y)$. Then locate the intersection of the axes
  in Figure~\ref{fig:threebox}.  The cell corresponding to that box is
  a desired divided cell.
\end{proof}
The N-cells described by Theorem~\ref{thm:Pitman} are called the
\emph{neighbors} of the G-cell in that theorem.  One of these
neighbors is characterized by $\absval{a_0/b_0}\geq2$; the other by
$\absval{b_1/a_1}\geq2$. Conversely, each of these inequalities allows
the construction of a neighboring G-cell of a given N-cell.

If both inequalities hold, then the N-cell serves
as an immediate link between consecutive reductions of the linear part
of $F_I(x,y)$.  However, it is also possible to find N-cells for which
neither of these will hold, so that they are not neighbors of a
G-cell.  Such cells will be considered in the next section.
 
The boxes shown in Figure~\ref{fig:threebox} give the matrices shown
in (\ref{eq:neighborplus}). The third column affects only the location
of a cell and not its shape, and is significant only for
describing cells having some particular relation to the original G-cell.

When $a_1\leq0$, the transition matrices of (\ref{eq:neighborplus})
are replaced by
\begin{equation}
\matthreer{1&1&0\\0&1&0\\0&0&1\\}
\textbox{ or }
\matthreer{1&0&0\\-1&1&0\\0&0&1\\}.
\label{eq:neighborminus}
\end{equation} 
(If $a_1=0$, its sign should be chosen opposite to the sign of $b_0$;
if $a_1=b_0=0$ the different constructions only involve cells with
degenerate boxes.)

\section{Superfluous Cells}
This section investigates the role of the I-reduced cells that are
neither G-cells nor neighbors of G-cells. We refer to such cells as
\emph{superfluous cells} for a reason that is given in
Theorem~\ref{thm:superfluous}.

Figure~\ref{fig:superfluous} shows the portion of a chain of divided
cells starting with a cell $C_-$ for which $-2<a_0/b_0<-1$ and
$b_1/a_1<-2$ (the values used when drawing the figure were
$-(3+\sqrt5)/4\approx-1.309016994$ and
$-(3+\sqrt5)\approx-5.236067977$). The figure also includes the inner
box of the first cell that is seen to also be the inner box of
\emph{all} cells shown. All but the last of these has a unique
successor, and the figure shows this chain of unique successors.  The
last cell shown, $C_+$, has $a_0/b_0<-2$, so there will be a choice of
possible successors, none of which are shown.  For all the cells $C$
shown in Figure~\ref{fig:superfluous}, the cell $C_+$ will be called
the \emph{forward anchor} of $C$ and $C_-$ will be called the
\emph{backward anchor} of $C$.

\begin{figure}
\begin{center}
  \includegraphics[height=3in]{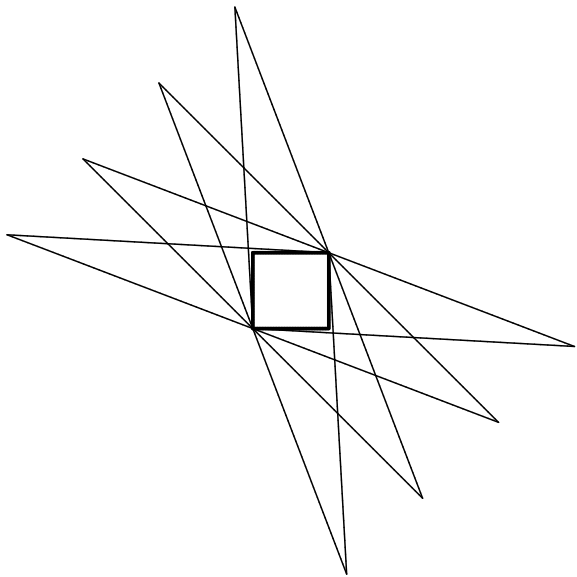}
 \end{center}
  \caption{  \label{fig:superfluous} Superfluous cells}
\end{figure}

\begin{thm}\label{thm:superfluous}For a superfluous
  cell, the anchors are uniquely determined. For every vertex of a
  superfluous cell, a vertex of one of the anchors is nearer in the
  extended partial order.
\end{thm}
\begin{proof} 
Since $1\leq\absval{a_0/b_0}\leq2$, the proof of
Theorem~\ref{thm:chain} shows that the divided cell algorithm involves
a unique successor that is also an N-cell.  As long as that cell is
superfluous, the algorithm generates a unique forward chain.  A closer
examination of the function giving $a_0/b_0$ for the successor in
terms of the corresponding quantity in the original cell is an
expansive mapping with $\pm1$ as fixed points.  From this, it follows
that, apart from degenerate cases, the chain starting from any
superfluous cell will reach a neighbor of a G-cell in a finite number
of steps (this is not difficult to show, but the details are awkward
to express, so they will be omitted).  The process stops at the
\emph{forward anchor} of the superfluous cell.  The process of
stepping backwards through the chain of divided cells is governed by
the ratio $b_1/a_1$ in the same way, leading to the \emph{backward
anchor} of the original cell.  Thus, Figure~\ref{fig:superfluous}
describes the \emph{only} way that superfluous cells can occur and
relates these cells to the anchors. A study of the explicit matrix
relating a superfluous cell to its successor shows that the anchors
are attached to consecutive reduced bases of the lattice.

Two of the vertices of a superfluous cell are also vertices of its
inner box.  These will be called the \emph{inner vertices} of the
cell.  The inner vertices are shared with all cells shown in
Figure~\ref{fig:superfluous} including the anchors, so they have now
been found in a non-superfluous cell.

The remaining vertices of the cells in Figure~\ref{fig:superfluous}
(the \emph{outer vertices}) lie on a lattice line parallel to and
adjacent to the line joining the inner vertices, and one of the
vertices of an anchor will be nearer in the extended partial order
than a given outer vertex of a superfluous cell.
\end{proof} 

\section{A rigorous Framework}The emphasis here has been visual.
Figures were used to illustrate the constructions and proofs.  These
figures were drawn using the \emph{Maple} Symbolic Computation
System.  In order to tell the system what to draw, the cells and boxes
were represented by matrices like $B$ of (\ref{eq:inhommat}). 

The visual approach was present in \cite{del47}, but was not used much
by subsequent authors.  Computers have facilitated the re-introduction
of graphics into exposition, including the use of color where
appropriate (the figures in this paper were presented in color at the
conference). At the same time, increasing fluency in the language of
Linear Algebra has encouraged the use of matrices to represent the
objects met in the study.  Our intent here was to use these
developments to present old results in a way that will encourage new
research.

Some weaknesses of the Divided Cell \emph{Algorithm} have appeared in
our exposition, but we have also shown that its application to
Inhomogeneous Diophantine Approximation can rely on methods like the
ordinary continued fraction that are associated to the Homogeneous
Markoff Spectrum.  Divided Cells become a tool for organizing the
subject rather than a device for computing properties of individual
problems.


\section{Acknowledgments}
We thank Takao Komatsu for all he did to make to make this conference
a success. We acknowledge the funding that he obtained for our
participation in the conference.

The inspiration to study divided cells came from Bill Moran.  We also
acknowledge the funding that brought us (separately) to Adelaide for
that work and apologize tor the long delay in producing the fruits of
that work.

\bibliographystyle{plain}  

\bibliography{bumby}

\begin{thebibliography}{1}

\bibitem{BarIV}
E.~S. Barnes.
\newblock The inhomogeneous minima of binary quadratic forms. {IV}.
\newblock {\em Acta Math.}, 92:235--264, 1954.

\bibitem{BarS-DIII}
E.~S. Barnes and H.~P.~F. Swinnerton-Dyer.
\newblock The inhomogeneous minima of binary quadratic forms. {III}.
\newblock {\em Acta Math.}, 92:199--234, 1954.

\bibitem{bumby1991}
Richard~T. Bumby.
\newblock The continued fraction algorithm approached through quadratic forms.
\newblock {\em Ranchi University Mathematical Journal}, 22:1--24, 1991.

\bibitem{sensual}
John~H. Conway.
\newblock {\em The sensual (quadratic) form}, volume~26 of {\em Carus
  Mathematical Monographs}.
\newblock Mathematical Association of America, Washington, DC, 1997.
\newblock With the assistance of Francis Y. C. Fung.

\bibitem{cusickandflahive89}
Thomas~W. Cusick and Mary~E. Flahive.
\newblock {\em The {M}arkoff and {L}agrange spectra}, volume~30 of {\em
  Mathematical Surveys and Monographs}.
\newblock American Mathematical Society, Providence, RI, 1989.

\bibitem{del47}
B.~N. Delone.
\newblock An algorithm for the ``divided cells'' of a lattice~(in {R}ussian).
\newblock {\em Izvestiya Akad. Nauk SSSR. Ser. Mat.}, 11:505--538, 1947.

\bibitem{pitman58}
Jane Pitman.
\newblock The inhomogeneous minima of a sequence of {M}arkov forms.
\newblock {\em Acta Arithmetica}, V:81--116, 1958.

\end{thebibliography}

\end{document}